\documentclass[12pt, reqno]{amsart}

\title[Center of Banach algebra valued Beurling algebras]{Center of Banach algebra valued Beurling algebras}

\makeatletter
\@namedef{subjclassname@2020}{\textup{2020} Mathematics Subject Classification}
\makeatother

\allowdisplaybreaks[1]

\author[B.~Talwar and R.~Jain]{Bharat Talwar and Ranjana Jain}

\address{Bharat Talwar, Department of  Mathematics, University of Delhi, Delhi}
\email{btalwar.math@gmail.com}

\address{Ranjana Jain, Department of Mathematics, University of Delhi, Delhi}
\email{rjain@maths.du.ac.in}

\thanks{Bharat Talwar is supported by Senior Research Fellowship of CSIR (File number 09/045(1442)/2016-EMR-I)}

\usepackage{relsize}
\usepackage{amsmath,amsfonts,amssymb,amsthm,mathrsfs}
\usepackage{hyperref}
\usepackage{cleveref}
\usepackage[margin=2.9cm]{geometry}
\usepackage{setspace}
\usepackage{xcolor}
\numberwithin{equation}{section}
\newtheorem{theorem}{\bf Theorem}[section]
\newtheorem{lemma}[theorem]{\bf Lemma}
\newtheorem{cor}[theorem]{\bf Corollary}
\newtheorem{remark}[theorem]{\bf Remark}

\newtheorem{defin}[theorem]{\bf Definition}

\newcommand{\seq}{\subseteq}

\newcommand{\obp}{\otimes^\gamma}
\newcommand{\oi}{\otimes^\lambda}

\newcommand{\ot}{\otimes}

\newcommand{\Z}{\mathcal{Z}}

\newcommand{\C}{\mathbb{C}}

\newcommand{\mcal}{\mathcal}

\newcommand{\A}{L^1(G,A)}
\newcommand{\B}{L^1_\omega(G,A)}

\newcommand{\CSS}{$C\sp \ast$-algebras}
\newcommand{\BA}{Banach algebra}

\begin{document}
\begin{abstract}
We prove that for a \BA \ $A$ having a bounded $\Z(A)$-approximate identity and for every $\bf[IN]$ group $G$ with weight $w$ which is either constant on conjugacy classes or  $w \geq 1$, $\Z\big(L^1_w(G) \obp A\big) \cong \Z(L^1_w(G)) \obp \Z(A)$.
As an application, we discuss the conditions under which $\Z\big(\B\big)$ enjoys certain Banach algebraic properties, for example, weak amenability, semisimplicity etc.
\end{abstract}
\maketitle
\section{Introduction}
For two algebras $A$ and $B$, $\Z(A) \ot \Z(B) = \Z(A \ot B)$, where $\Z(C)$ denotes the center of algebra $C$.
If $A$ and $B$ are \BA s, then it is natural ask whether $\Z(A \obp B)$ is isometrically isomorphic to $\Z(A) \obp \Z(B)$, $\obp$ being the Banach space projective tensor product.
It is known to be true if $A$ and $B$ are \CSS \ (see \cite[Theorem 5.1]{GJ2}) and if $A = B = L^1(G)$ for any $\bf [FC]^-$ group $G$ (see \cite[Lemma 2.1]{Zhang}).
Note that \cite{Zhang} generalizes the results of \cite{Alaghmandan, Azimifard} and in these three articles the major focus was on studying the amenability and weak amenability properties.
The idea behind the proofs given in \cite{Zhang, Azimifard} is to use a projection from $L^1(G)$ onto $\Z(L^1(G))$.
Ingenious construction of one such projection is given in \cite{Zhang} which is somewhat different from the usual averaging technique used while working with $\bf [FIA]^-$ groups.
We used this technique in \cite[Theorem 4.13]{GJT2} and provided an affirmative answer to this question if $A$ is a unital \BA \ and $B = L^1(G)$, for specific classes of groups.
Analogues of a couple of results on $\Z(L^1(G))$ by Mosak \cite{mosak, mosak1} were also obtained in \cite{GJT2} for the center $\Z\big(\A\big)$ of generalized group algebras.

In this article we generalize all the results discussed in the preceding paragraph, giving relatively simpler proofs, by working in a more generalized setting of the $A$-valued Beurling algebras $\B$.
In particular, we drop some restrictions on $G$ and $A$ as imposed in \cite[Lemma 4.4, Theorem 4.7, Theorem 4.13]{GJT2} and  obtain a similar desirable description for $\Z\big(\B\big)$.
After a series of technical and interesting results, we present \Cref{centerdistributewhenweightFC} as main result of this paper which claims that $\Z\big(L^1_w(G) \obp A\big) \cong \Z(L^1_w(G)) \obp \Z(A)$ when $G$ is an $\bf[IN]$ group (a group having a neighbourhood of identity which is invariant under inner automorphisms of $G$),  $A$ has a bounded $\Z(A)$-approximate identity (see  \Cref{Condition(P)}) and weight $w$ is either constant on conjugacy classes or greater than or equal to 1  throughout $G$.
As an application, we discuss a few structure theoretic questions for  $\Z\big(\B\big)$.

\section{Different Identifications of Center}

Let $G$ be a locally compact group with identity $e$.
Then $(G, \mcal{B}, dx)$ is a measure space, where $\mcal{B}$ is the Borel $\sigma$-algebra and $dx$ is the left Haar measure.
Weight $w$ on $G$ is a measurable positive function such that $w(xy) \leq w(x)w(y)$ for every $x,y \in G$.
In view of \cite[Theorem 3.7.5]{Reiter}, $w$ can be assumed to be continuous.
For any \BA \  $A$, consider the set $$\text{L}^1_w(G,A)= \Big\{ f: G \to A : f \text{ is $\mcal{B}$-measurable and }\int_G \| f(x)\|w(x) dx < \infty \Big\}.$$
Let $f,g \in \text{L}^1_w(G,A)$.
Then $(f \ast g)(x) = \int_G f(xy)g(y^{-1}) dy$ and $\|f \|_{w,A} = \int_G \| f(x)\|w(x) dx$ define a multiplication and a seminorm on $\text{L}^1_w(G,A)$, respectively.
When $A  = \C$, we write $\|f\|_{1,w}$ for $\|f \|_{w,A}$.
The set $\B$ of all equivalence classes determined by this seminorm becomes a \BA \ known as $A$-valued Beurling algebra.
As is customary, we will treat the elements of $\B$ as functions.
For any $a \in A$, $x,y \in G$ and $f \in \B$, define $(x \cdot f)(y) = f(x^{-1}y)$, $(f \cdot x)(y)= f(yx)$, $(fa)(x) = f(x)a$ and $(af)(x) = a f(x)$.
It is easy to check that all these elements belong to $\B$.

A part of the following result is proved in \cite[Lemma 2.7]{DedaniaMathToday}.
One can also have a look at \cite[Lemma 3.2, Lemma 3.3]{GJT2}, where an analogous statement is proved for $\A$.
The fact that $w$ is locally bounded \cite[Lemma 1.3.3]{kaniuthbook} along with some necessary changes in that proof can be used to obtain the following.
\begin{lemma}\label{Dedania}
Let $f \in \B$ and $y \in G$.
\begin{enumerate}
\item The maps $G \ni x \to x \cdot f, f \cdot x \in \B$ are continuous.
\item  $\| y \cdot f\|_{w,A} \leq w(y) \|f\|_{w,A}$ and $\| f \cdot y \|_{w,A} \leq w(y^{-1})\Delta(y^{-1}) \|f\|_{w,A}$.
\end{enumerate}
\end{lemma}
We first provide an analogue of the characterization of center of convolution algebra as given in \cite[Proposition 1.2]{mosak1}.
If $f$ is a function form $G$ to $A$, we define $\|f\|(x):= \|f(x)\|$ for every $x \in G$.

\begin{lemma}\label{centre1Beur}
Let $G$ be a locally compact group	and $A$ be a Banach algebra.
Then, $$\Z\big(\B\big) =  \{ f \in \B : \Delta(s^{-1})(f\cdot s^{-1})a =  a(s\cdot f ), \forall s \in G, a \in A\}.$$
\end{lemma}
\begin{proof}
Let $f \in \Z\big(\B\big)$.
For any $a \in A$, $s \in G$ and any compact symmetric set $U \in \mcal{B}$, write $a_{sU} = a \chi_{sU}$.
It follows from local boundedness of $w$ that $a_{sU} \in \B$.
Let $\epsilon >0$.
It suffices to show that $\|\Delta(s^{-1})(f\cdot s^{-1})a -  a(s\cdot f )\|_{w,A} < \epsilon$.
	Following the calculations in the proof of \cite[Theorem 3.4]{GJT2}, we obtain that
	\begin{align*}
	\| m(Us^{-1}) (f \cdot s^{-1}) a - f \ast a_{sU} \|_{w,A} & \leq \int_{U s^{-1}} \int_G \| f \cdot s^{-1} a - (f \cdot x) a \|(y)w(y) dy dx \\
	& \leq m(Us^{-1}) \sup_{x \in U} \| f \cdot s^{-1} - (f \cdot s^{-1}) \cdot x^{-1}\|_{w,A} \|a\|,
	\end{align*}
	
	and
	\begin{align*}
\|m(U) a(s \cdot f) -  a_{sU} \ast f \|_{w,A} &
\leq \Delta(s) \| a\| \int_U \int_G \| (s \cdot f)(y) - f(x^{-1}s^{-1} y) \|w(y) dy dx \\
&\leq  m(U) \|a\| \sup_{x \in U} \| s \cdot f -  s \cdot (x \cdot f)\|_{w,A}\\
& \leq  m(U) \|a\| w(s) \sup_{x \in U} \|f -  (x \cdot f)\|_{w,A}.
	\end{align*}
	So, 
	\begin{align*}
	\| \Delta(s^{-1}) & (f\cdot s^{-1})a - a(s\cdot f ) - \frac{1}{m(U)} (f \ast a_{sU} - a_{sU} \ast f)\|_{w,A} \\& \leq  \big(\Delta(s^{-1})\sup_{x \in U} \| f \cdot s^{-1} - (f \cdot s^{-1}) \cdot x^{-1}\|_{w,A} + w(s) \sup_{x \in U} \| f - x \cdot f\|_{w,A}\big) \|a\|.
	\end{align*}
The result now follows from \Cref{Dedania}.

For the converse, it is sufficient to prove that $f \ast g = g \ast f$ for every continuous function $g$ with compact support.
This can be proved exactly as in \cite[Theorem 3.4]{GJT2}, by proving that in every neighbourhood of $f \ast g - g \ast f$ there is an element which belongs to the  set $\mathrm{span}\{ \Delta(s)^{-1} (f \cdot s^{-1}) a - a(s \cdot f): a \in A, s \in G\}$.
\end{proof}
With this in hand, the following analogue of \cite[Lemma 4.3]{GJT2} can be given with a few adjustments in its proof.

\begin{lemma}\label{centervaluedfunctionsincenterBeur}
Let $G$ be a locally compact group and $A$ be a Banach algebra.
Then,
	\[
	\Z\big(\B\big) \subseteq L^1_w(G, \Z(A)).
	\]
\end{lemma}
\begin{proof}
Consider a non-zero element $f$  in $\Z\big(\B\big)$. 
Suppose there exists a Borel set $E$ of positive and finite measure such that $f(x) \notin \Z(A)$ for every $x \in E$.

By \Cref{centre1Beur} we obtain $f a = a f$ for every $a \in A$. 
Let $\mcal{B}_{E}$ denote the $\sigma$-algebra consisting of all Borel sets contained in $E$. 
Then, for any $F \in \mcal{B}_{E}$, we find that
\begin{eqnarray*}
\left( \int_F f(x)w(x)\, dx\right)a &=& \int_F f(x)w(x)a\, dx =  \int_F \big((fa)w\big)(x)\, dx\\
& = & \int_F \big((af)w\big)(x)\, dx = a \left(\int_F f(x)w(x)\, dx\right)
\end{eqnarray*}
for all $a\in A$, i.e., $\int_F f(x)w(x) dx \in \Z(A)$ for all $F \in \mcal{B}_E$.
Define $H: \mcal{B}_E \rightarrow \Z(A)$ by $H(F) = \int_F f(x)w(x) dx$.
Then, $H$ is an $m$-continuous (i.e., $\lim_{m(F) \to 0}H(F) = 0$) vector measure of bounded variation (\cite[Theorem II.2.4]{diestel}).
Thus, by \cite[Corollary III.2.5]{diestel}, there exists a $g \in L^1(E,\Z(A))$ such that $H(F) = \int_F g(x) dx$ for all $F \in \mcal{B}_E$. 
This shows that $\int_F (f(x)w(x) - g(x)) dx = 0$ for every $F \in \mcal{B}_{E}$; so that $fw = g$ a.e. on $E$, by \cite[Corollary II.2.5]{diestel}. 
Since $g(E) \subseteq \Z(A)$ and $w(E) \seq (0, \infty)$, this is a contradiction to the existence of $E$. 
Hence, $f(x) \in \Z(A)$ for almost every $x \in G$.
\end{proof}
\begin{remark}
Unlike in $L^1(G)$ \cite{mosak},  it can be seen from \Cref{centervaluedfunctionsincenterBeur} that for $\Z\big(\B\big)$ to be non-trivial it is not sufficient that $G$ be an ${\bf[IN]}$ group.
This is because $\Z(A)$ might be trivial.
However, if $G$ is an ${\bf[IN]}$ group and $A$ is a \BA \ with non-trivial center, then $\Z\big(\B\big) \neq \{0\}$.
To prove this, choose a compact neighbourhood $E$ of $e$ which is invariant under inner automorphisms.
Then $\chi_E \in \Z(L^1(G))$ \cite{mosak}.
Since $w$ is locally bounded, we have $\chi_E \in L^1_w(G)$.
Using the fact that $G$ is unimodular, we obtain from  \Cref{centre1Beur} that $\chi_Ea \in \Z\big(\B\big)$ for every $a \in \Z(A)$.
\end{remark}
We will now present some necessary conditions for $\Z(\B)$ to be non-trivial.
Note that $G$ being an $\bf [IN]$ group is not a necessity as can be  demonstrated by taking $A$ to be a \BA \ with trivial multiplication, in which case $\Z(\B) \neq \{0\}$ no matter which group $G$ one takes.
To get rid of such pathological examples, mathematicians generally work with \BA s having a bounded approximate identity.
However, we only need the following relaxed condition.

\begin{defin}\label{Condition(P)}
A $\Z(A)$-approximate identity of a \BA \ $A$ is a net $\{a_\alpha\}$ in $A$ such that $a_\alpha a \to a$ for every $a \in \Z(A)$.
If, in addition, the net $\{a_\alpha\}$ is bounded, then we call it a bounded $\Z(A)$-approximate identity.
\end{defin}
Note that  a \BA \ $A$ has a bounded $\Z(A)$-approximate identity if either $A$ or $\Z(A)$ has a left or right bounded approximate identity.
Now onward, the \BA \ $A$ will be assumed to have a bounded $\Z(A)$-approximate identity $\{a_\alpha\}$.
\begin{lemma}\label{PropertyP}
Let $G$ be a locally compact group	and $A$ be a Banach algebra.
If $0 \neq f \in L^1_w(G,\Z(A))$, then the net $\{a_\alpha f\}$ in $\B$ converges to $f$.
\end{lemma}
\begin{proof}
Let $\{a_{\alpha} \}$ be bounded by $M$. 
For any $\epsilon >0$, since $L^1_w(G) \ot \Z(A)$ is dense in $L^1_w(G,\Z(A))$ \cite[Theorem 2.2]{Ebrahim}, there exists $f' = \sum_{i=1}^r  f_i \ot a_i \in L^1_w(G) \ot \Z(A)$ such that $\| f- f' \|_{w,A} < \epsilon$.
As $a_\alpha a_i \to a_i$ for every $1 \leq i \leq r$, we can choose $\alpha$ such that $\| a_\beta a_i - a_i \| \leq \epsilon/( \sum_{i=1}^r \| f_i\|_{1,w})$ for every $1 \leq i \leq r$ and $\beta \geq \alpha$.
For every $\beta \geq \alpha$, 
\begin{align*}
	\| a_\beta f' - f' \|_{w,A} & = \big\| \sum_{i=1}^r f_i \ot (a_\beta a_i - a_i) \big\|_{w,A}  \leq \sum_{i=1}^r \| f_i \|_{1,w} \| a_\beta a_i - a_i \| \\ & \leq \sum_{i=1}^r \| f_i \|_{1,w} \big(\epsilon/( \sum_{i=1}^r \| f_i\|_{1,w})\big) \leq \epsilon.
\end{align*}
Hence,
\begin{align*}
\| a_\beta f - f \|_{w,A} & \leq \| a_\beta f -  a_\beta f' \|_{w,A} + \|a_\beta f' - f' \|_{w,A} + \| f' - f \|_{w,A} \\ & \leq \| a_\beta \| \epsilon +  \epsilon + \epsilon < (M+2)\epsilon.
\end{align*}
This proves the result.
\end{proof}

Techniques of \cite{mosak} are used to prove the latter half of following.
\begin{lemma} \label{centrenew}
Let $G$ be a locally compact group with weight $w$ which is  either constant on conjugacy classes or $w \geq 1$ and $A$ be a Banach algebra.
Then $G$ is an ${\bf[IN]}$ group whenever $\Z\big(\B\big) \neq \{0\}$.
\end{lemma}
\begin{proof}
Let  $0 \neq f \in \Z\big(\B\big)$.
For any $x \in G$, $x \cdot f, f \cdot x^{-1} \in L^1_w(G, \Z(A))$, from 
\Cref{centervaluedfunctionsincenterBeur}.
Thus, by \Cref{PropertyP},  \( \Delta(x^{-1})(f \cdot x^{-1})a_\alpha \to \Delta(x^{-1}) f \cdot x^{-1}  \text{ and }   a_\alpha (x \cdot f) \to (x \cdot f).\)
Hence, in $L^1_w(G)$, we have $$\| \Delta(x^{-1}) (f \cdot x^{-1})a_\alpha \| \to \Delta(x^{-1})\| f \cdot x^{-1} \| = \Delta(x^{-1})\| f \| \cdot x^{-1}, \text{ and}$$ $$\| a_\alpha (x \cdot f)\| \to \| (x \cdot f)\| =  x \cdot \| f\|.$$
	
Case(i) $w \geq 1$:
It follows from \Cref{centre1Beur} that $0 \neq \| f\| \in \Z(L^1_w(G))$, which in turn implies that $G$ is an $\bf [IN]$ group, as $w\geq 1$ (see, \cite{LiukkonenBeur}).
	
Case (ii) $w$ is constant on conjugacy classes:
The proof of \cite[Lemma 4.4]{GJT2} works here except for the trivial modifications we now describe.
Put $h(x) = \| f(x) w(x) \|^{1/2}$.
Then using \Cref{centre1Beur} and \Cref{PropertyP}, we obtain $h(txt^{-1}) = \| f(txt^{-1}) w(txt^{-1}) \|^{1/2} = \| (t^{-1} \cdot f \cdot t^{-1})(x) w(x) \|^{1/2} = \| \Delta(t)^{1/2}f(x) w(x) \|^{1/2} = \Delta(t)^{1/2} h(x)$.
Now the continuous function $p(s) = \int_G h(sy)h(y)dy$ will give us a compact neighbourhood of $e$ which is invariant under inner automorphisms.
\end{proof}
The restrictions on weight in  the previous result are not artificial.
In fact, if $G$ is an abelian group with weight $w$, then there is an equivalent weight $\tilde{w} \geq 1$ on $G$ \cite[Lemma 3.2]{DedaniaMathToday} such that  $L^1_w(G)$ and $L^1_{\tilde{w}}(G)$ are isomorphic as \BA s.
Also, every weight on an abelian group is trivially constant on conjugacy classes.
Moreover, if $G$ is a compact group, then $w \geq 1$ \cite[Corollary 1.3.4]{kaniuthbook}.

Before presenting our main results, let us derive some consequences of what we have obtained so far.
In the rest of the article, $\mathrm{Inn}(G)$ shall denote the group of all inner automorphisms ($\mathrm{Ad}_y(x) = y^{-1}xy$) of $G$.
For any function $f$ on $G$, define $(\mathrm{Ad}_y \cdot f)(x) = f(yxy^{-1})$ for every $x,y \in G$.
\begin{lemma}\label{bettercharacterizationofcenter}
Let $G$ be  an $\bf [IN]$ group and $A$ be a Banach algebra.
Then $$\Z\big(\B\big) = \{ f \in \B : \mathrm{Ad}_y \cdot f = f \text{ and } fa = af \  \forall  a \in A, y \in G \}.$$
\end{lemma}
\begin{proof}
Note that for every $f \in \B$ we have $(\mathrm{Ad}_y \cdot f )(x)  = (y^{-1} \cdot f \cdot y^{-1})(x)$.
Using \Cref{centre1Beur} and the fact that every $\bf [IN]$ group is unimodular, we obtain
\begin{equation}\label{star}
\Z\big(\B\big) = \{ f \in \B : a (\mathrm{Ad}_y \cdot f) = f a \ \forall a \in A, y \in G \}.
\end{equation}	
	
	If $f \in \B$ is such that $fa = af$ and $\mathrm{Ad}_y \cdot f = f$ for every $y \in G$ and $a \in A$, then it readily follows that $a (\mathrm{Ad}_y \cdot f) = a f = f a$ for every $y \in G$ and $a \in A$.

To prove the other containment, let $0 \neq f \in \Z\big(\B\big)$.
Taking $y=e$ in \Cref{star}, we obtain $fa = af$ for every $a \in A$.

From \Cref{centervaluedfunctionsincenterBeur} it follows that $f, \mathrm{Ad}_y \cdot f \in L^1_w(G,\Z(A))$ for every $y \in G$.
So, by \Cref{PropertyP} we obtain \( 0 =  \big(a_\alpha(\mathrm{Ad}_y \cdot f) - f a_\alpha \big) \to (\mathrm{Ad}_y \cdot f) - f\) for every $y \in G$.
Hence,  $\mathrm{Ad}_y \cdot f = f$ for every $y \in G$.
\end{proof}
\begin{cor}\label{centerBeru2}
Let $G$ be  an $\bf [IN]$ group and $A$ be a Banach algebra.
We have
\begin{align}\label{CenterNice}
\Z\big(\B\big) = \{ f \in L^1_w(G, \Z(A)) : \mathrm{Ad}_y \cdot f = f \ \ \forall y \in G\}.
\end{align}
In particular, if $\Z(A)$ has a bounded approximate identity, then
\begin{align*}
\Z\big(\B\big) = \Z\big(L^1_w(G, \Z(A))\big) \text{ and } \Z\big(L^1_w(G) \obp A\big) \cong \Z\big(L^1_w(G) \obp \Z(A)\big).
\end{align*}
\end{cor}
\begin{proof}
First statement is a direct consequence of \Cref{centervaluedfunctionsincenterBeur} and \Cref{bettercharacterizationofcenter}.
The second statement follows from \Cref{centervaluedfunctionsincenterBeur} and the fact that $\Z(A)$ has a bounded $\Z(\Z(A))$-approximate identity.
The third statement is a consequence of the well-known fact that $L^1_w(G) \obp A \cong \B$ \cite[Theorem 2.2]{Ebrahim}.
\end{proof}
Let $G$ be an $\bf [IN]$ group and  $w$ be a weight  which is constant on conjugacy classes.
Just as in \cite{GJT2}, we consider the $\sigma$-subalgebra $\mcal{B}_{\mathrm{inv}}=\{B \in \mcal{B}: \mathrm{Ad}_y(B) = B\ \text{for all } y \in G\}$ and define the corresponding \BA \ $L^1_{w,{\mathrm{inv}}}(G,A)$ arising from $(G, \mcal{B}_{\mathrm{inv}}, dx_{\mathrm{inv}} = dx_{|_{\mcal{B}_{\mathrm{inv}}}})$.
If $f \in L^1_{w,\mathrm{inv}}(G,A)$, then $f$ is $\mcal{B}_{\mathrm{inv}}$-measurable and hence $\mcal{B}$-measurable \cite[Lemma 4.6]{GJT2}.
Clearly, $L^1_{w,\mathrm{inv}}(G,A) \seq \B$.
If $w$ is a weight such that $w \geq 1$, then we may define $$L^1_{w,\mathrm{inv}}(G,A) = \big\{ f \in L^1_w(G, A) \ | \ f \text{ is $\mcal{B}_{\mathrm{inv}}$-measurable} \big\}.$$
In both these cases, from \cite[Lemma 4.6]{GJT2}, we have 
$$L^1_{w,\mathrm{inv}}(G,A) =\{f \in \B:  f \text{ is constant on the conjugacy  classes of G}\}.$$
Note as in \cite[Theorem 2.2]{Ebrahim} that $L^1_{w,\mathrm{inv}}(G,A)  \cong L^1_{w,\mathrm{inv}}(G) \obp A$ for such weights.

\begin{theorem}\label{centerdistributewhenweightFC}
Let $G$ be an $\bf [IN]$ group with a weight $w$ which is either constant on conjugacy classes or $w \geq 1$ and $A$ be a \BA.
Then $$\Z\big(\B\big) \cong \Z(L^1_w(G)) \obp \Z(A).$$
\end{theorem}
\begin{proof}
We first claim that $\Z(L^1_w(G)) = L^1_{w,\mathrm{inv}}(G)$.
	
Let us first assume that $w$ is constant on conjugacy classes.	
We know that $\Z(L^1(G))$ has a uniformly bounded approximate identity -see proof of \cite[Corollary 1.6]{LiukkonenGroup}.
As convolution of an $L^1$ function with an $L^\infty$ function gives a continuous function, $\Z(L^1(G)) \cap C(G)$ is dense in $\Z(L^1(G))$.
For a fixed $f \in \Z(L^1_w(G))$, $fw \in \Z(L^1(G))$ \cite{mosak1}, so, $fw$ is approximated in $L^1(G)$ by a sequence $\{g_n\} \in \Z(L^1(G)) \cap C(G)$.
Since both $g_n$ and $w$ are constant on conjugacy classes, we obtain that $f$ is constant on conjugacy classes.
So, $f \in L^1_{w,\mathrm{inv}}(G)$, proving that $\Z(L^1_w(G)) \seq L^1_{w,\mathrm{inv}}(G)$.
It now follows from \Cref{centerBeru2} that $\Z(L^1_w(G)) = L^1_{w,\mathrm{inv}}(G)$.

If $w\geq 1$, then $L^1_{w, \mathrm{inv}}(G) \seq L^1_w(G) \seq  L^1(G)$.
Thus, being constant on conjugacy classes, every element of  $L^1_{w, \mathrm{inv}}(G)$ is contained in  $\Z(L^1(G))$.
This further implies that $L^1_{w, \mathrm{inv}}(G)\seq \Z(L^1_w(G))$.
Conversely, if $f \in \Z(L^1_w(G))$, then $f \in L^1_w(G)$ and $f \cdot x = x^{-1} \cdot f$ for every $x \in G$ since $G$ is unimodular.
This proves that $f \in \Z(L^1(G))$ and hence $f$ is constant on conjugacy classes, proving the claim.

Thus, in both the cases, $$L^1_{w,\mathrm{inv}}(G,\Z(A)) \cong L^1_{w,\mathrm{inv}}(G) \obp \Z(A) \cong \Z(L^1_w(G)) \obp \Z(A).$$ 
From \Cref{centerBeru2}, $L^1_{w, \mathrm{inv}}(G,\Z(A)) \seq \Z\big(\B\big)$, so we only need to  check the reverse containment.
Further, since an arbitrary $f \in \Z\big(\B\big)$ is a member of $L^1_w(G,\Z(A))$, it is sufficient to show that $f$ is $\mcal{B}_{\mathrm{inv}}$-measurable.

As $f$ is $\mcal{B}$-essentially separable valued \cite[Proposition 2.15]{ryan}, there exists $E \in \mcal{B}$ with zero measure such that $f(E^c)$ is contained in a separable space.
We can redefine $f$  to be zero on $E$
and hence $f$ is $\mcal{B}_{\mathrm{inv}}$-essentially separable valued.
In view of \Cref{centerBeru2}, for every $\phi \in A^\ast$ and $y \in G$ we have $Ad_y \cdot (\phi \circ f)(x) = (\phi \circ f)(yxy^{-1}) = \phi (f(yxy^{-1})) = \phi (f(x))$ for almost every $x \in G$.
Thus, $\phi \circ f \in \Z(L^1_w(G)) = L^1_{w,\mathrm{inv}}(G)$ and hence $f$ is weakly $\mcal{B}_{\mathrm{inv}}$-measurable.
It now follows from \cite[Proposition 2.15]{ryan} that $f$ is $\mcal{B}_{\mathrm{inv}}$-measurable.
\end{proof}
The following result generalizes \cite[Theorem 4.13]{GJT2} (by taking  $w \equiv 1$).
It also generalizes \cite[Lemma 2.1]{Zhang} fully and \cite[Proposition 2.2]{Zhang} to some extent.
\begin{cor}
Let $G$ be an $\bf [IN]$ group and $A$ be a \BA.
	Then, $$\Z\big(\A\big) \cong \Z(L^1(G)) \obp \Z(A).$$
\end{cor}
We note the following interesting consequences of \Cref{centerdistributewhenweightFC} without giving the definitions of the concepts discussed as they are not being used rigorously.
We would also like to point out that (2) and (3) below generalize \cite[Proposition 2.3]{Zhang}.
\begin{cor}\label{consequences}
Let $G$ be an $\bf [IN]$ group with weight $w \geq 1$ and $A$ be a \BA. Then, 
\begin{enumerate}
\item $\Z\big(\B\big)$ is semisimple if and only if $\Z(A)$ is semisimple.
\item If $\Z(L^1_w(G))$ and $\Z(A)$ are weakly amenable, then so is $\Z\big(\B\big)$.
\item If $\Z\big(\B\big)$ is weakly amenable and semisimple, then both $\Z(L^1_w(G))$ and $\Z(A)$ are weakly amenable.
\item If $\Z\big(\B\big)$ is Tauberian and semisimple, then  both $\Z(L^1_w(G))$ and $\Z(A)$ are Tauberian.
\item If both $\Z(L^1_w(G))$ and $\Z(A)$ are Tauberian then so is $\Z\big(\B\big)$.
\item $\Z(\B)$ is regular if and only if both $\Z(L^1_w(G))$ and $\Z(A)$ are regular.
\item  $\Z\big(\B\big) = \{ f \in L^1_w(G,\Z(A)): f \text{ is constant on conjugacy classes}\}.$
\end{enumerate}
\end{cor}
\begin{proof}
(1) We know that $L^1(G)$ has approximation property \cite[page no. 325]{kaniuthbook}.
So, $L^1_w(G)$ has approximation property because the map $f \in L^1_w(G) \to fw \in L^1(G)$ is a Banach space isomorphism.
Hence, the natural map from $L^1_w(G) \obp A$ to $L^1_w(G) \oi A$, the Banach space injective tensor product, is injective \cite[Theorem A.2.12]{kaniuthbook}.
Thus, its restriction from $\Z\big(L^1_w(G) \obp A\big) = \Z(L^1_w(G)) \obp \Z(A)$  to $\Z(L^1_w(G)) \oi \Z(A) \seq L^1_w(G) \oi A$ is also injective from the injectivity of $\oi$.
Since $\Z(L^1_w(G))$ is a semisimple and commutative  \BA \ \cite[Corollary 2.3.7]{Rickart}, the result now follows from \cite[Theorem 2.11.6]{kaniuthbook}.

\smallskip
(2)  It follows from the fact that projective tensor product of weakly amenable commutative \BA s is weakly amenable (see \cite[Proposition 2.8.71]{Dales}).

\smallskip 
(3) Since $\Z(L^1_w(G))$ and $\Z(A)$ are semisimple, there exists multiplicative linear functionals $\phi_1$ and $\phi_2$ on $\Z(L^1_w(G))$ and $\Z(A)$, respectively, see \cite[Definition 2.1.9]{kaniuthbook}.
Then $\phi_1 \obp 1_{\Z(A)} : \Z(L^1_w(G)) \obp \Z(A) \to \Z(A)$ and $1_{\Z(L^1_w(G))} \ot \phi_2 : \Z(L^1_w(G)) \obp \Z(A) \to \Z(L^1_w(G))$ are surjective homomorphism and  \cite[Proposition 2.8.64]{Dales} gives the result.

\smallskip
(4) This follows from (1) and  \cite[Lemma 2.1]{Tewari}.

\smallskip
(5) Since both $\Z(L^1_w(G))$ and $\Z(A)$ are commutative, the result  follows from \cite[Theorem 1.~$\mathrm{P}_2$]{Gelbaum62}.

\smallskip
(6) It follows from \Cref{centerdistributewhenweightFC} and \cite[Theorem 3]{Tomiyama59}.

\smallskip
(7) This is same as saying that $\Z\big(\B\big) = L^1_{w,\mathrm{inv}}(G,\Z(A))$, which is proved in \Cref{centerdistributewhenweightFC}.
\end{proof}
\begin{remark}
Note that the hypothesis $w \geq 1$ is merely used to obtain semisimplicity of $\Z(L^1_w(G))$.
The results (2), (5), (6) and (7) are true if $w$ is constant on conjugacy classes.
\end{remark}

\end{document}